\documentclass[11pt]{amsart}
\usepackage[margin=1.17in]{geometry}

\usepackage{amssymb}
\usepackage{amsthm}
\usepackage{amsmath}
\usepackage{mathrsfs}
\usepackage{amsbsy}
\usepackage[all]{xy}
\usepackage{bm}
\usepackage{hyperref}
\usepackage{tikz}
\usepackage{array}
\usepackage{float}
\usepackage{enumerate}
\usepackage{xcolor}
\usepackage{hhline}
\usepackage{mathtools}
\allowdisplaybreaks
\usepackage[noadjust]{cite}

\usepackage{caption}
\usepackage{subcaption}
\usepackage{diagbox}
\usepackage{tikz}
\usepackage{bbm}
\usepackage{booktabs}

\usepackage[noabbrev,capitalise]{cleveref}

\newenvironment{enumerate*}%
{\begin{enumerate}[(I)]%
\setlength{\itemsep}{10pt}%
\setlength{\parskip}{0pt}}%
{\end{enumerate}}

\newtheorem{theorem}{Theorem}[section]
\newtheorem{proposition}[theorem]{Proposition}
\newtheorem{corollary}[theorem]{Corollary}
\newtheorem{conjecture}[theorem]{Conjecture}

\newtheorem{lemma}[theorem]{Lemma}

\theoremstyle{definition}

\DeclareMathOperator{\Sym}{Sym}

\title{Graham conjecture on small sets in abelian groups}

\author[]{Simone Costa}
\address[Simone Costa]{DICATAM, Universit\`a degli Studi di Brescia}
\email{simone.costa@unibs.it}
\author[]{Stefano Della Fiore}
\address[Stefano Della Fiore]{DII, Universit\`a degli Studi di Brescia}
\email{stefano.dellafiore@unibs.it}
\author[]{Mattia Fontana}
\address[Mattia Fontana]{Department of Engineering and Sciences, Universitas Mercatorum}
\email{mattia.fontana@studenti.unimercatorum.it}
\author[]{Lluís Vena}
\address[Lluís Vena]{Department of Mathematics,
Universitat Politècnica de Catalunya}
\email{lluis.vena@upc.edu}


\begin{document}

\begin{abstract}
A famous conjecture of Graham asserts that every set $A \subseteq \mathbb{Z}_p \setminus \{0\}$ can be ordered so that all partial sums are distinct.
Although this conjecture was recently proved for sufficiently large primes by Pham and Sauermann in~\cite{PM} (combined with earlier results of \cite{BBKMM}), it remains open for general abelian groups, even in the cyclic case $\mathbb{Z}_k$.

In this paper, using a recursive approach, we investigate the sequenceability of subsets $A$ in generic abelian groups for small values of $|A|$. We prove that any subset $A \subseteq G\setminus\{0\}$ with $|A| \leq 20$ is sequenceable where previously it was known only for $|A|\leq 9$. This bound is improved to $|A| \leq 22$ for zero-sum subsets. Finally, regarding the related CMPP conjecture, we show that zero-sum subsets without inverse pairs are sequenceable for $|A| \leq 23$.
\end{abstract}
\maketitle

\section{Introduction}
Let $A$ be a finite subset of an abelian group $(G,+)$. We say that an ordering
$a_1, \ldots, a_{|A|}$ of $A$ is \emph{valid} if its partial sums
$p_1=a_1, p_2=a_1+a_2, \ldots, p_{|A|}=a_1 + \cdots + a_{|A|}$ are pairwise distinct.
Moreover, this ordering is \emph{sequencing} if it is valid and $p_i \neq 0$
for every $1 \le i \le |A|-1$. A set is said to be \emph{sequenceable} if it admits a sequencing ordering.

In the literature, there are several conjectures about valid  and sequencing orderings.
We refer to \cite{CMPP18, OllisSurvey, PD} for an overview of the topic,
\cite{AKP, AL20, ADMS16, CDOR} for lists of related conjectures,
and \cite{BFMPY} for a treatment using rainbow paths.

Here, we first recall the conjecture Graham posed for groups $\mathbb{Z}_p$ (and more generally, posed for abelian groups by Alspach).

\begin{conjecture}[Graham/Alspach conjecture, see \cite{GR}, \cite{EG} and \cite{AL20}]\label{conj:main}
Let $G$ be an abelian group. Then every subset $A \subseteq G \setminus\{0\}$ is sequenceable.
\end{conjecture}
The following related conjecture was posed by Costa, Morini, Pasotti and Pellegrini keeping in mind applications to Heffter arrays.
\begin{conjecture}[CMPP conjecture, see \cite{CMPP18}]\label{conj:main2}
Let $G$ be an abelian group. Then every subset $A \subseteq G \setminus\{0\}$ whose elements sum to zero and such that $\{x,-x\}\not\subset A$ for any $x\in G$, is sequenceable.
\end{conjecture}

Until recently, the main results on these conjectures were for small values of $|A|$;
in particular, in~\cite{CDOR}, the Graham's conjecture was proved, using the Combinatorial Nullstellensatz theorem, for sets $A\subseteq \mathbb{Z}_p$ of size at most $12$.
For generic abelian groups $G$ the best result is due to Alspach and Liversidge (\cite{AL20}) who prove, via a brute-force poset-approach, the conjecture for sets $A\subseteq G\setminus\{0\}$ of size at most $9$. 
Regarding the CMPP conjecture, it was proved for sets $A\subseteq \mathbb{Z}_p$ of size at most $13$ and for sets $A\subseteq G\setminus\{0\}$ of size at most $10$.

The first result involving arbitrarily large sets $A$ was presented by Kravitz~\cite{NK},
who used a rectification argument to show that Graham's conjecture holds for all sets $A$ of size
$|A|\le \log p/\log\log p$. 
Sawin~\cite{Sawin} also proposed a similar argument (but was not published). This bound was then improved in \cite{BK} by Bedert and Kravitz, and, very recently, by Costa and Della Fiore in \cite{CD26}, who proved, again using a rectification based approach that:
\begin{theorem}[\cite{CD26}]\label{thm:mainBK}
Let $p$ be the least prime divisor of $k$. Then there exists a constant $c>0$ such that every subset $A \subseteq \mathbb{Z}_k \setminus\{0\}$
is sequenceable provided that
\[
|A| \leq \exp\!\big(c(\log p)^{1/3}\big).
\]
\end{theorem}

As mentioned above, the most impressive result on sequenceability problems was proved in~\cite{PM}, where, combined with earlier results of \cite{BBKMM}, they settled Graham's conjecture on $\mathbb{Z}_p$ for all sufficiently large primes $p$. This result is a consequence of anticoncentration inequalities developed using a discrete Fourier approach that seems hard to adapt to the cyclic case. Moreover, the condition that $p$ is sufficiently large leaves open all small primes.

In this paper, we investigate Conjectures \ref{conj:main} and \ref{conj:main2} for small values of $|A|$ and in generic abelian groups $G$. In particular, using a recursive approach, we have been able to prove the following.
\begin{theorem}\label{thm:main}
Let $G$ be an abelian group. Then every subset $A \subseteq G\setminus\{0\}$
is sequenceable provided that
\[
|A| \leq 20.
\]
\end{theorem}
If we assume, moreover, that $\sum_{x\in A}x=0$, we can improve this bound.
\begin{theorem}\label{sumzero}
Let $G$ be an abelian group and let $A \subseteq G\setminus\{0\}$ be such that $\sum_{x\in A}x=0$.
Then $A$ is sequenceable provided that
\[
|A| \leq 22.
\]
\end{theorem}
Finally, in case of the CMPP conjecture, we can go one size further.
\begin{theorem}\label{sumzeroxmx}
Let $G$ be an abelian group and let $A \subseteq G\setminus\{0\}$ be such that $\sum_{x\in A}x=0$ and $\{x,-x\}\not\subset A$ for any $x\in G$.
Then $A$ is sequenceable 
provided that
\[
|A| \leq 23.
\]
\end{theorem}

The results obtained here significantly improve the previously known bounds for general abelian groups. 
Indeed, while the conjecture was previously verified only for sets of size at most $9$ in arbitrary abelian groups, 
our recursive approach allows us to extend this bound up to $20$, and even further in the presence of additional 
structural assumptions.

The key ingredient of our method is a Recursive Lemma, presented in Section $2$ as Corollary~\ref{col:recursive_lemma}, allowing the reduction of the problem to smaller sets via 
a suitable merging operation. In this section, we will first present a polynomial proof of this lemma for the groups $\mathbb{Z}_p$ and then a more general direct proof. The general proof relies on the characterization of the subsets of $G$ closed under the sum of distinct elements, which is also given in Section~\ref{sec:rec_lemma}. 
Finally, in the last Section, we will discuss and describe the computational approach used to get our main theorems.
\section{The Recursive Lemma} \label{sec:rec_lemma}
In this section, we 
prove a Recursive Lemma that 
is crucial in our recursive approach for the case of a sum-zero subset $A\subset\mathbb{Z}_k$. Namely, we 
prove that we can choose $x_1,x_2\in A$ so that $x_1+x_2$ is nonzero and does not belong to $A\setminus\{x_1,x_2\}$. This allows us to work recursively with $A':=(A\setminus\{x_1,x_2\})\cup \{x_1+x_2\}.$ 
\subsection{A Combinatorial Nullstellensatz's approach in $\mathbb{Z}_p$}

We present a polynomial-based proof of the Recursive Lemma. Although we will later prove this lemma in a more general context, we believe this particular approach is elegant and mathematically significant.

The proof relies on the following version of Alon’s Combinatorial Nullstellensatz:
\begin{theorem}\cite[Theorem 1.2]{Alon}\label{thm:alon}
Let $\mathbb{F}$ be a field and let $f=f(x_1,\ldots,x_k)$ be a polynomial in $\mathbb{F}[x_1,\ldots,x_k]$. Suppose the degree of $f$ is 
$\sum\limits_{i=1}^k t_i$, 
where each $t_i$ is a nonnegative integer, and suppose the coefficient of
$\prod\limits_{i=1}^k x_i^{t_i}$ in $f$ is nonzero.
Then, if $A_1,\ldots,A_k$ are subsets of $\mathbb{F}$ with $|A_i|> t_i$, there are $a_1 \in A_1,\ldots,a_k \in A_k$ so that 
$f(a_1,\ldots,a_k)\neq 0$.
\end{theorem}

In order to apply this theorem for proving the existence of the elements $x_1,x_2$ that can be merged, we construct a suitable homogeneous polynomial $F_k$ of degree $\frac{(k+2)(k-1)}{2}$, identifying a monomial with 
nonzero coefficient such that the degree of each of its terms $x_i$ is less than $|A|=k$. Here we will assume that $p>25$ and $23\geq k\geq 10$ as the case $p\leq 25$ has already been solved in \cite{CMPP18}, the case $k\leq 9$ has been solved in \cite{AL20}, and our recursive approach can reach $k$ up to $23$. 
We define: 
$$F_k(x_1,\dots,x_k)=(x_1+x_2)\prod_{i=3}^k(x_1+x_2-x_i)\prod_{i,j\in [1,k], i>j} (x_i-x_j).$$

The existence of values $x_1,\ldots,x_k\in A$ such that
$F_k(x_1,\ldots,x_k)\neq 0$, given by Theorem \ref{thm:alon}, implies that
$x_1+x_2$ is nonzero and it does not belong to $A\setminus \{x_1,x_2\}$.

We can use Alon's combinatorial Nullstellensatz  to prove this statement. Indeed, assuming $p>25$ we find, with Mathematica, for any $k\in \{10,\dots,23\}$ a monomial of the form $\prod_{i=1}^k x_i^{t_i}$ in the development of $F_k$ whose coefficient is nonzero and such that $k>t_i$ for all $i\in \{1,\dots,k\}$ (see Table \ref{tab:tab1} for the explicit coefficients). We obtain then the Recursive Lemma in the $\mathbb{Z}_p$ case.

\begin{lemma}
Let $A\subseteq \mathbb{Z}_p\setminus \{0\}$ be of cardinality $10\leq k\leq 23$. Then there exist two distinct elements $x_1,x_2\in A$ so that $A':=(A\setminus\{x_1,x_2\})\cup \{x_1+x_2\}$ is a subset of $\mathbb{Z}_p\setminus \{0\}$ of size $k-1$.
\end{lemma}

\begin{table}[h]
\centering
\small
\begin{tabular}{c c}
\hline
$k$ & Coefficient in $\mathbb{Z}$ \\
\hline
10  & $2^2 \cdot 11$ \\
11  & $2 \cdot 3^3$ \\
12  & $-5 \cdot 13$ \\
13  & $-7 \cdot 11$ \\
14  & $2 \cdot 3^2 \cdot 5$ \\
15  & $2^3 \cdot 13$ \\
16  & $-7 \cdot 17$ \\
17  & $-3^3 \cdot 5$ \\
18  & $2^3 \cdot 19$ \\
19  & $2 \cdot 5 \cdot 17$ \\
20  & $-3^3 \cdot 7$ \\
21  & $-11 \cdot 19$ \\
22  & $2 \cdot 5 \cdot 23$ \\
23  & $2^2 \cdot 3^2 \cdot 7$ \\
\hline
\end{tabular}
\caption{Coefficient of the monomial $x_1^{k-1} x_2 x_3^{2} \cdots x_k^{k-1}$ in $F_k$ for $9 \le k \le 23$.}
\label{tab:tab1}
\end{table}

\subsection{A direct proof}
In this subsection we 
prove the Recursive Lemma in a much more general setting. More precisely we 
characterize the sets $A$ for which $a_i+a_j$ belongs to $A\cup \{0\}$ whenever $i$ and $j$ are distinct indices. Proposition~\ref{prop:another_2} shows that, for most cases, one of the elements can be chosen freely.

The characterization, Theorem~\ref{prop:another_3}, and the next proposition, Proposition~\ref{prop:another_2}, follows from the same ideas leading to Kneser's theorem \cite{Kneser_1,Kneser_2,david}, which claims that if $G$ is an abelian group, then $|A+B|\geq |A|+|B|-1$ unless $A+B$ is a union of cosets of a subgroup.

\begin{proposition}[Translation by $a_i$]\label{prop:another_2}
	Let $A=\{a_1,\ldots,a_k\}\subset \mathbb{Z}_m\setminus \{0\}$ be a set of pairwise different elements. Then, for each $i\in [k]$, one of the following holds.
	\begin{enumerate}
		\item\label{enum:1} There exists a $j\neq i$ such that $a_i+a_j\notin \{0,a_1,\ldots,a_k\}$.
        \item \label{enum:4} There exists an $s<k$ such that $A\cup \{0\}=\{a_i,0,-a_i,-2a_i,\ldots,-sa_i\}\cup M$, where $M$ is a union of cosets of the subgroup generated by $a_i$. If $M$ is the empty set, $s=k-1$.
		\end{enumerate}
\end{proposition}

\begin{proof}[Proof of Proposition~\ref{prop:another_2}]
	Let $g_i:\mathbb{Z}_m\to \mathbb{Z}_m$ be the mapping $g_i(x)=x+a_i$. Then $|g_i(\{a_1,\allowbreak \ldots,a_k\})|=k$. Further, since $0\notin \{a_1,\ldots,a_k\}$, then $a_i\notin g_i(\{a_1,\ldots,a_k\})$. Let $b_0,\ldots,b_s$ be the longest chain of distinct elements of $\{a_1,\ldots,a_k\}$ so that $g_i(b_{j})=b_{j+1}$ for each $j\in\{0,\ldots,s-1\}$.

    If $g_i(b_s)\in \{a_1,\ldots,a_k\}$, then $g_i(b_s)=b_0$ as the map $g_i$ is injective, and we obtain $g_i(b_s)=b_s+a_i=b_0$ and since $b_s=g_i^s(b_0)=sa_i+b_0$, then $(s+1)a_i+b_0=b_0$ which implies that $(s+1)a_i=0$, and since $s+1$ is the minimal with such property, then $s+1$ is the order of $a_i$. Further, if we let $S=\langle a_i\rangle$ be the subgroup generated by $a_i$, then $\{b_0,\ldots,b_s\}+S=\{b_0,\ldots,b_s\}$  and it is a coset of $S$. Note that $\{b_0,\ldots,b_s\}\subset \{a_1,\ldots,a_k\}$ and thus it is not the $0$ coset (or the subgroup $S$), as it does not contain the identity element $0$.
We consider the new set $\{a_1,\ldots,a_k\}\setminus \{b_0,\ldots,b_s\}$ and apply the result inductively.

Observe that this procedure of iteratively applying $g_i$ cannot reach $a_i$, since the mapping is bijective, and that would mean that we apply $g_i$ to $0$, but $0$ is not in $A$, and thus we never consider it as an element in the domain of the mapping. In particular, $s\leq k-2$.
If $g_i(b_s)=0$, and $s=k-2$. Then we conclude that $b_0+a_i=b_1$, $b_1+a_i=b_2$ and in general $b_j=b_0+ja_i \to b_0=b_j-ja_i$ with $b_s=b_0+(k-2)a_i=0 \to b_s-(k-2)a_i=b_0$, and $b_s+a_i-(k-1)a_i=b_0=-(k-1)a_i$. Then $A=\{a_1,\ldots,a_k\}=\{b_0,\ldots,b_s,a_i\}$ as $s=k-2$, and $\{b_0,\ldots,b_s,a_i\}=\{-(k-1)a_i,\ldots,-a_i,a_i\}$ and so the second part of the statement follows with $M$ as the empty set.

If $g_i(b_s)=0$, and $s<k-2$, this means that $A\setminus \{a_i,b_0,\ldots,b_s\}$ is non-empty. Observe also that $g_i(A\setminus \{b_0,\ldots,b_s\})\cap \{b_0,\ldots,b_s\}=\emptyset$ since the mapping $g_i$ is bijective and all the elements in $\{b_1,\ldots,b_s\}$ have a preimage already, and $g_i(x)\neq b_0$ for each $x\in A\setminus \{a_i,b_0,\ldots,b_s\}$ by the assumed maximality on the $s$ of $\{b_0,\ldots,b_s\}$. By applying the same procedure of finding a (new) maximal chain $\{c_0,\ldots,c_r\}$ of the elements of $A\setminus \{a_i,b_0,\ldots,b_s\}$;
we conclude that we are either in the first case ($g_i(c_r)=c_0$) and find another coset of the subgroup, and we would again apply the result recursively, or $c_r+a_i$ is a new element different from $0$ and the first part of the statement follows. (Note that $g_i(b_s)=0$ already, and thus the other two cases do not apply).
\end{proof}

The Recursive Lemma 
is then a consequence of the following result.

\begin{theorem}[Characterization Theorem]\label{prop:another_3}
	Let $G$ be an abelian group and let $A=\{a_1,\ldots,a_k\}$ $\subseteq G\setminus \{0\}$ be a set of $k\geq 2$ elements such that, for any distinct $i,j$, $a_i+a_j\in \{0,a_1,\ldots,a_k\}$. Then $A\cup \{0\}$  is a subgroup of $G$.
\end{theorem}

\begin{proof}[Proof of Theorem~\ref{prop:another_3}]
First of all, we consider the case where $A$ is constituted only by involutions. In this case, it is clear that $A\cup \{0\}$ is inverse-closed and it is closed under the sum of its elements and is hence a subgroup of $G$. Similarly, also when $|A|=2$ we have that $A\cup \{0\}$ is inverse-closed and it is closed under the sum. So in the following we will assume the existence of a non-involution element and that $|A|\geq 3$.

Now we prove that, under these hypotheses, there exists a non-involution $a_i\in A$ so that $-a_i\in A$.
Indeed, assume otherwise. With the hypothesis for any distinct $i,j$, $a_i+a_j\in \{0,a_1,\ldots,a_k\}$ we conclude that for any distinct $\{i,j\}$, $a_i+a_j\not=0$. There are ${k \choose 2}$ such pairs and only $n$ elements in $A$, so each element of $A$ must be the sum of an average of  $$\frac{{k \choose 2}}{k}=(k-1)/2$$ distinct pairs.
If a sum would have more than this number of pairs, then two such pairs would have an element in common, thus making the two pairs the same. So each element of $A$ must be the sum of exactly $(k-1)/2$ distinct pairs. Note that this is not possible if $k$ is even. For the case $k$ odd, for each $i$ there is a matching $\{(r_i,s_i)\}_{i\in[(k-1)/2]}$ between elements in $\{a_1,\ldots,a_k\}\setminus a_i$ such that $a_{r_j}+a_{s_j}=a_i$. By considering the relation of the non-involution element $a_t$ and its matched element $a_{t'}$ and $a_i$ so that (so $a_t+a_{t'}=a_i$ with $(t,t')=(r_j,s_j)$ for some $j\in [(k-1)/2]$), and the relation $a_{t'}=a_i+a_{i'}$ for the $i'$-th element matching the $i$-th, we conclude that $a_t+a_{i'}=0$. Since $a_t$ was a non-involution element, this proves that there exists a non-involution $a_t\in A$ so that $-a_t\in A$.


Now, assuming that $a_i,-a_i\in A$ with $a_i\neq -a_i$, we prove that $2a_j\in A\cup \{0\}$ for all $a_j\not=\pm a_i$. Indeed, both $a_j+a_i$ and $a_j-a_i\in A\cup \{0\}$ and, since these elements are distinct, also
$$(a_j+a_i)+(a_j-a_i)=2a_j\in A\cup \{0\},$$
and claim follows.
Now, since $2a_j=a_j+a_j\in A\cup \{0\}$, it follows that, for all $a_j\not=\pm a_i$, the sets 
$$\{0+a_j,a_1+a_j,\dots,a_k+a_j\}\subseteq A\cup \{0\}.$$
Since both these sets have cardinality $k+1$, we must have that 
$$\{0+a_j,a_1+a_j,\dots,a_k+a_j\}= A\cup \{0\}$$
and hence $-a_j\in A$ for all $j$ (also when $j=i$).

Finally, we consider $2a_i$ and want to see that $2a_i\in A\cup \{0\}$. We assume $a_1\not=\pm a_i$, and we recall that we have 
$\{0+a_1,a_1+a_1,\dots,a_k+a_1\}= A\cup \{0\}$.
Therefore, there exists $j$ such that $a_i=a_1+a_j$. 
If $a_j= a_i$, then we have $a_1=0$ a contradiction, and if $a_j= -a_i$ then $a_1=2a_i$ as desired. If $a_j\not=\pm a_i$, then 
this means that
$$2a_i=a_i+a_i=(a_j+a_1)+a_i=a_j+(a_1+a_i).$$
Note that there exists an index $h$ such that $a_1+a_i=a_i+a_1=a_h\in A$ which implies that 
$$2a_i=a_j+(a_1+a_i)=a_j+a_h=a_h+a_j \in A\cup \{0\}$$
where the last inclusion holds since $a_j\not=\pm a_i$. We also note that the same argument (exchanging $a_i$ by $-a_i$) applies for $2(-a_i)=-2a_i\in A$.
We have proved that
$2a_j\in A\cup\{0\}$ for all $j$ (also when considering $\pm a_i$). It follows that $A\cup \{0\}$ is both inverse-closed and closed under the sums of its elements and thus is a subgroup.
 \end{proof}
\begin{corollary}[Recursive Lemma]\label{col:recursive_lemma}
Let $G$ be an abelian group, and let $A=\{a_1,\ldots,a_k\}\subseteq G\setminus \{0\}$ be a set of $k$ elements. Then we either see that $A$ is sequenceable or that there exist distinct $\{i,j\}$ such that $a_i+a_j\not \in A\cup \{0\}$.
\end{corollary}
\proof
Let us assume that $a_i+a_j \in A\cup \{0\}$ for any distinct $i,j$. Then, due to Theorem \ref{prop:another_3}, we either have that $k=1$ in which case $A$ is trivially sequenceable or that $A=H\setminus \{0\}$ for some subgroup $H$ of $G$. In the latter case we have that $A$ is sequenceable due to Theorem 1.2 of \cite{AKP}.
\endproof
\section{Code description}
We describe the algorithm used to establish Propositions~~\ref{prop1} and~\ref{prop2}.
The method is a proof by contradiction: we assume that
$A=\{a_1,\dots,a_k\}\subseteq G\setminus\{0\}$ is a \emph{counterexample},
meaning that every ordering of $A$ has at least one forbidden collision among
its partial sums.

\subsection*{Partial sums and incidence vectors}

Let $x=(x_1,\dots,x_k)$ be an ordering of $A$ and define partial sums
\[
y_0=0,\qquad y_i=\sum_{j=1}^i x_j \quad(1\le i\le k).
\]
A collision $y_s=y_t$ with $0\le s<t\le k$ is equivalent to a consecutive
zero-sum block
\begin{equation}\label{eq:block}
	x_{s+1}+x_{s+2}+\cdots+x_t \;=\; 0.
\end{equation}
Fix a labeling $A=\{a_1,\dots,a_k\}$ and write the ordering as a permutation
$\omega\in\Sym(k)$, so that $x_r=a_{\omega(r)}$.
Given an interval of positions $[s+1,t]$, we encode which labels appear in that
block by the \emph{incidence vector} $v\in\{0,1\}^k$, defined by $v_i=1$ if
$a_i$ occurs somewhere in the block and $v_i=0$ otherwise.  Then
\eqref{eq:block} becomes the label-based relation
\begin{equation}\label{eq:incidence_relation}
	\sum_{i=1}^k v_i\,a_i \;=\; 0.
\end{equation}
Along any branch of the search tree we collect these vectors as rows of a
binary matrix $C\in\{0,1\}^{m\times k}$, where $m$ is the number of recorded
intervals on that branch.  Any integer combination of the rows of $C$ gives a
further linear relation among the $a_i$'s.

\subsection*{Search tree structure}

A \emph{node} of the search tree is a pair $(\omega,C)$ consisting of a current
ordering $\omega$ and the constraint matrix $C$ collected along the branch.
The root is $(\mathrm{id},\emptyset)$, where $\mathrm{id}=(1,2,\dots,k)$.

To generate children of a node $(\omega,C)$, the algorithm considers every
interval of positions $[i,j]$ (with positions indexed from $0$ to $k-1$)
satisfying
\[
1 \le i < j \le k
\qquad\text{and}\qquad
j-i \le \left\lfloor\tfrac{k}{2}\right\rfloor.
\]
The length restriction $j-i\le\lfloor k/2\rfloor$ is used in the zero-sum case;
for general (non-zero-sum) subsets we exclude only the full interval $[0,k-1]$.

For each admissible interval $[i,j]$ the algorithm:
\begin{itemize}
	\item builds its incidence vector $v\in\{0,1\}^k$ by setting
	$v_{\omega(\ell)-1}=1$ for $\ell\in[i,j]$ and $0$ elsewhere;
	\item skips this interval if $v$ already appears as a row of $C$
	(adding a duplicate row gives no new information);
	\item performs one adjacent swap at a boundary of $[i,j]$ to produce a new
	ordering $\omega'$:
	if $i\ge 1$, swap positions $i$ and $i-1$ (left boundary move), and if $i=0$,
	swap positions $j$ and $j+1$ (right boundary move);
	\item creates the child node $(\omega',\,C\cup\{v\})$.
\end{itemize}
Each child receives its own copy of the augmented matrix, so different branches
are independent.

The tree is explored breadth-first (BFS).  A node is not expanded further once
a terminal certificate (described next) is detected.

\subsection*{Terminal certificates}

At each node $(\omega,C)$ the algorithm checks whether the row span of $C$ over
$\mathbb{Q}$ contains one of a small list of vectors.  The check is always the
same: for a candidate $u\in\mathbb{Q}^k$, compare
\[
\text{rank}(C)\quad\text{and}\quad \text{rank}\!\begin{pmatrix}C\\u\end{pmatrix}.
\]
If the two ranks coincide, then $u$ lies in the row span of $C$ and a
contradiction is obtained.

\begin{itemize}
	\item \emph{Zero-element certificate.}
	If $e_i$ lies in the row span of $C$, then clearing denominators gives an
	integer combination of recorded relations equal to $a_i=0$, contradicting
	$A\subseteq G\setminus\{0\}$.
	
	\item \emph{Equality certificate.}
	If $e_i-e_j$ lies in the row span of $C$ for some $i\ne j$, then we obtain
	$a_i=a_j$, contradicting that the elements of $A$ are distinct.
	
	\item \emph{Compression certificate.}
	By the Recursive Lemma there exist $a_1,a_2\in A$ such that
	\[
	A':=(A\setminus\{a_1,a_2\})\cup\{a_1+a_2\}
	\]
	has size $k-1$ and admits a sequencing $x_1,\dots,x_{k-1}$.
	This fixed sequencing implies that certain consecutive blocks in $A'$ are
	\emph{not} zero-sum.  These forbidden blocks can be listed as a finite set of
	binary vectors $\mathtt{initial\_cons}\subset\{0,1\}^k$.
	If any $w\in\mathtt{initial\_cons}$ lies in the row span of $C$, then
	the counterexample assumption forces a relation $\sum_i w_i a_i=0$ of a type
	that cannot occur in the compressed sequencing, giving a contradiction.
\end{itemize}

All three checks are implemented using the same rank test, computed with
standard numerical linear algebra (\texttt{numpy.linalg.matrix\_rank}).

\subsection*{Correctness and completeness}

The algorithm is correct in the following sense: when a terminal certificate is
found, it is a purely linear consequence of the recorded relations
\eqref{eq:incidence_relation}, so it yields a contradiction independently of
the specific abelian group $G$.

It is complete in the following sense: if the BFS exploration finishes with
every branch either reaching a terminal certificate or having no children, then
the counterexample assumption cannot be sustained.  In particular, no set $A$
of the given size can be a counterexample, so every such set is sequenceable.

\subsection*{Complexity and implementation}

At depth $d$ a node has $m=d$ recorded rows.  The number of candidate intervals
is at most $\binom{k}{2}$, reduced by the length restriction and by the
deduplication check.  The rank test costs $O(m^2k)$ per candidate, so a rough
per-node bound is $O(k^2 m^2 k)=O(m^2 k^3)$.  For the values of $k$ considered
here ($k\le 23$), terminal certificates are typically reached after very few
levels, so the practical running time is much smaller than worst-case bounds.

\subsection*{Example  for $k=3$}

We give a minimal example where a contradiction is obtained after recording only
two interval constraints.  Let $A=\{a_1,a_2,a_3\}\subseteq G\setminus\{0\}$ with
pairwise distinct elements.  Start from the root node with ordering
\[
\omega=(1,2,3),\qquad C=\emptyset,
\]
and assume (as in the proof by contradiction) that \emph{every} ordering of $A$
contains a forbidden consecutive zero-sum block.

Choose the interval $[i,j]=[0,1]$ in the ordering $(1,2,3)$.  Its incidence
vector is
\[
v^{(1)}=(1,1,0),
\]
encoding the relation $a_1+a_2=0$.  Since $i=0$, the algorithm performs the
right boundary move and swaps positions $j$ and $j+1$, i.e.\ positions $1$ and
$2$.  This produces the new ordering
\[
\omega'=(1,3,2),
\]
and the child node has constraint matrix $C=\{v^{(1)}\}$.

From this child, choose the interval $[1,2]$ in the ordering $(1,3,2)$.  The
labels in these positions are $\{3,2\}$, so the incidence vector is
\[
v^{(2)}=(0,1,1),
\]
encoding the relation $a_2+a_3=0$.  Now the matrix of recorded constraints is
\[
C=\{(1,1,0),(0,1,1)\}.
\]
At this point we obtain a terminal certificate of \emph{equality} type, since
\[
v^{(1)}-v^{(2)}=(1,1,0)-(0,1,1)=(1,0,-1)=e_1-e_3.
\]
Equivalently, $e_1-e_3$ lies in the row span of $C$, so the algorithm stops.
The corresponding contradiction in $G$ is
\[
(a_1+a_2)-(a_2+a_3)=0 \quad\Longrightarrow\quad a_1-a_3=0 \quad\Longrightarrow\quad a_1=a_3,
\]
which contradicts the assumption that the elements of $A$ are distinct.
\newline

\section{Conclusion}

By running the code \footnote{All the computations were implemented in Python and run on a local server (with 2 TB of RAM and 128 AMD CPU cores). All code used in this work is openly available at the following repository:  \href{https://doi.org/10.5281/zenodo.18997904}{10.5281/zenodo.18997904}.} of the previous section we obtain the following statement. 
\begin{proposition}\label{prop1}
    Let $A'$ be a subset of cardinality $k-1$ consisting of nonzero elements, of an abelian group $G$, summing to zero, 
and suppose $A'$ admits a sequencing $x_1, x_2, \ldots, x_{k-1}$. 
Given two elements $a_1, a_2$ such that
$$
A = \{a_1,\, a_2,\, x_2,\, x_3,\, \ldots,\, x_{k-1}\}
$$
is a set of cardinality $k$ of nonzero elements satisfying $a_1 + a_2 = x_1$, 
we conclude that $A$ admits a sequencing in the following cases:
\begin{itemize}
    \item $k \leq 22$, under the assumption that $\sum_{a \in A} a = 0$;
    \item $k \leq 23$, under the additional assumption that $A$ contains no pair 
    of the form $\{x, -x\}$.
\end{itemize}
\end{proposition}
Therefore, by virtue of the Recursive Lemma Theorems 
~\ref{sumzero}  and ~\ref{sumzeroxmx} follow readily.

In the case where $A$ is not zero-sum, we proceed analogously; however, in this case, we can not assume that the merged elements sum to $x_1$ since the ordering is not cyclic and we have to consider $k-1$ cases. Anyway, by means of the computational search, we obtain the following proposition.

\begin{proposition}\label{prop2}
    Let $A'$ be a subset of cardinality $k-1$ consisting of nonzero elements, of an abelian group $G$ and suppose $A'$ admits a sequencing $x_1, x_2, \ldots, x_{k-1}$. 
Given two elements $a_1, a_2$ such that
$$
A = \{a_1,\, a_2,\, x_1,\, x_2,\, \ldots,\,x_{i-1},x_{i+1},\, \ldots,\, x_{k-1}\}
$$
is a set of cardinality $k$ of nonzero elements satisfying $a_1 + a_2 = x_i$ for some $i$ in $[1,k-1]$, 
we conclude that $A$ admits a sequencing whenever $k\leq 20$.
\end{proposition}
Therefore, also in this case, by the Recursive Lemma Theorem \ref{thm:main} follows.

To stress the effectiveness of our recursive approach, it is worth noting that a pure tree-algorithm (i.e. considering only the first two terminal certificates), without the Recursive Lemma, is limited to sets of size $|A| \leq 10$, consistently with the current literature (see \cite{AL20}, which arrives at size $9$).

\section*{Acknowledgements}
The first and second authors were partially supported by INdAM--GNSAGA. The last author has been supported by the I+D+i project PID2023-147202NB-I00 funded by MICIU/AEI/10.13039/501100011033.

\end{document}